\documentclass[bibalpha]{amsart}
\usepackage{amssymb}
\usepackage{textcomp}
\usepackage[mathscr]{euscript}
\usepackage{pb-diagram}
\usepackage{enumitem}
\usepackage{comment}

\newtheorem{theorem}{Theorem}[section]

\newtheorem{prop}[theorem]{Proposition}

\newtheorem{fact}[theorem]{Fact}

\newtheorem{lemma}[theorem]{Lemma}

\theoremstyle{definition}

\theoremstyle{remark}

\ProvideTextCommandDefault{\cprime}{(U+042C)}

\newcommand{\st}{\operatorname{st}}

\newcommand{\cl}{\operatorname{Cl}}

\newcommand{\Sh}[1]{\ensuremath{\mathscr{#1}^{\mathrm{Sh}}}}

\newcommand{\nip}{\mathrm{NIP}}

\newcommand{\Cal}[1]{\ensuremath{\mathcal{#1}}}
\newcommand{\Sa}[1]{\ensuremath{\mathscr{#1}}}

\newcommand{\R}{\mathbb{R}}

\begin{document}
\title[]{An nip structure which does not interpret an infinite group but whose Shelah expansion interprets an infinite field}

\author{Erik Walsberg}
\address{Department of Mathematics, Statistics, and Computer Science\\
Department of Mathematics\\University of California, Irvine, 340 Rowland Hall (Bldg.\# 400),
Irvine, CA 92697-3875}
\email{ewalsber@uci.edu}
\urladdr{http://www.math.illinois.edu/\textasciitilde erikw}

\date{\today}

\maketitle

\begin{abstract}
We describe one.
\end{abstract}
\section{Introduction}
\noindent All structures are first order and ``definable" means ``first order definable, possibly with parameters".
Let $\Sa M$ be a structure.
The structure induced on $A \subseteq M^m$ by $\Sa M$ is the structure with domain $A$ whose primitive $n$-ary relations are all sets of the form $X \cap A^n$ for $\Sa M$-definable $X \subseteq M^{nm}$.
Let $\Sa N$ be a highly saturated elementary extension of $\Sa M$.
The \textbf{Shelah expansion} $\Sh M$ of $\Sa M$ is the structure induced on $M$ by $\Sa N$.
A subset of $M^n$ is \textbf{externally definable} if it is of the form $X \cap M^n$ for some $\Sa N$-definable $X \subseteq N^n$.
Saturation shows that the collection of externally definable sets does not depend on choice of $\Sa N$, so $\Sh M$ essentially does not depend on choice of $\Sa N$.
Fact~\ref{fact:shelah} is due to Shelah~\cite{Shelah-external}, see also Chernikov and Simon~\cite{CS-I}.

\begin{fact}
\label{fact:shelah}
Suppose $\Sa M$ is $\nip$.
Then every $\Sh M$-definable set is externally definable.
\end{fact}

\noindent Fact~\ref{fact:shelah} implies that $\Sh M$ is $\nip$ when $\Sa M$ is $\nip$.
More generally and informally, it shows that $\Sh M$ has the same combinatorial properties as $\Sa M$.
We show that new algebraic structure can appear in $\Sh M$.
Geometric stability theory contains dichotomies between combinatorial simplicity and algebraic structure.
Our construction suggests that if one seeks to obtain such dichotomies in the $\nip$ setting then one should look for the algebraic structure in the Shelah expansion.
If $\Sa M$ is stable then every externally definable set is definable, so this phenomenon cannot occur in the stable setting.
In \cite[Section 15]{big-nip} we described an $\nip$ expansion of an ordered abelian group which does not interpret an infinite field, but whose Shelah expansion does. 
The present example is similar.

\subsection{Acknowledgements} Thanks to Artem Chernikov for prompting us to think this through.

\section{Notation and Conventions}
\noindent Throughout $s,t$ are real numbers.
An open set in a topological space is \textbf{regular} if it is the interior of its closure.
A subset of an o-minimal structure is \textbf{independent} if it is independent in the sense of algebraic closure (equivalently: definable closure).
We let $\cl(X)$ be the closure in $\R^n$ of $X \subseteq \R^n$.

\section{The structure}

\noindent Let $\Sa R$ be an o-minimal expansion of $(\R,<,+)$, $\Sa R \prec \Sa N$ be highly saturated, $H$ be a dense independent subset of $N$, and $\Sa H$ be the structure induced on $H$ by $\Sa N$.
Dolich, Miller, and Steinhorn~\cite{DMS-Indepedent} study the expansion of an o-minimal structure by a unary predicate defining a dense independent set.
Fact~\ref{fact:independent} is \cite[2.16]{DMS-Indepedent}.

\begin{fact}
\label{fact:independent}
Any subset of $H^n$ definable in $(\Sa N,H)$ is of the form $X \cap H^n$ for some $\Sa N$-definable $X \subseteq N^n$.
It follows that the theory of $\Sa H$ is weakly o-minimal.
\end{fact}

\noindent Fact~\ref{fact:independent} shows that $\Sa H$ is $\nip$, furthermore dp-minimal and distal.
Informally: $\Sa H$ should have the same combinatorial properties as $\Sa R$.

\begin{prop}
\label{prop:H-no-interpret}
$\Sa H$ does not interpret an infinite group.
\end{prop}

\begin{proof}
A theorem of Eleftheriou~\cite[Theorem C]{Elef-small-sets} shows that $\Sa H$ eliminates imaginaries, so it suffices to show that $\Sa H$ does not define an infinite group.
Fact~\ref{fact:independent} and the fact that $H$ is independent together show that the algebraic closure in $\Sa H$ of any $A \subseteq H$ is $A$.
It is now routine to show that $\Sa H$ does not define an infinite group.
\end{proof}

\noindent Proposition~\ref{prop:H-interpret} below shows that $\Sh H$ interprets $(\R,<,+)$ in general, and interprets $(\R,<,+,\cdot)$ when $\Sa R$ expands $(\R,<,+,\cdot)$.
To obtain interpretability of $\Sa R$ we will need Lemma~\ref{lem:regular}.
It is easier to show that $(\R,<,+)$ or $(\R,<,+,\cdot)$ is interpretable, as 
$$ \{ (s,s',t) \in \R^3 : s + s' < t\} \quad \text{and} \quad \{ (s,s',t) \in \R^3 : ss' < t \}  $$
are both regular open.

\begin{lemma}
\label{lem:regular}
Every $\Sa R$-definable set is a boolean combination of regular open $\Sa R$-definable sets.
\end{lemma}

\noindent The proof is complicated by the fact that a definable open set need not be a union of finitely many open cells.
We say that $\Sa R$ defines a global field structure if there are definable $\oplus,\otimes: \R^2 \to \R$ such that $(\R,<,\oplus,\otimes)$ is isomorphic to $(\R,<,+,\cdot)$.
If $\Sa R$ defines a global field structure then the topological study of $\Sa R$-definable sets entirely reduces to the study of definable sets in o-minimal expansions of $(\R,<,+,\cdot)$.

\begin{proof}
An application of o-minimal cell decomposition shows that every $\Sa R$-definable set is a boolean combination of open $\Sa R$-definable sets, so it suffices to suppose $U \subseteq \R^n$ is open and $\Sa R$-definable and show that $U$ is a finite union of regular open $\Sa R$-definable sets.
Note that an open cell is regular. \newline

\noindent Edmundo, Eleftheriou, and Prelli~\cite{EEL-covering} show that if $\Sa R$ does not define a global field structure then $U$ is a finite union of open cells.
Suppose $\Sa R$ defines a global field structure.
Without loss of generality we suppose $\Sa R$ expands $(\R,<,+,\cdot)$.
Let $B_t$ be the open ball in $\R^n$ with center the origin and radius $t > 0$.
It suffices to show that $U \cap B_2$ and $U \setminus \cl(B_1)$ are both finite unions of definable regular open sets.
Let $\iota : \R^n \setminus \{0\} \to \R^n \setminus \{0\}$ be the definable homeomorphism given by
$$ \iota(t_1,\ldots,t_n) = (t^{-1}_1,\ldots,t^{-1}_n). $$
Wilkie~\cite{Wilkie-covering} shows that any bounded definable open set is a finite union of open cells.
So $U \cap B_2$ is a finite union of open cells.
Furthermore $\iota(U \setminus \cl(B_1)) \subseteq B_1$ is a union of open cells $V_1,\ldots,V_m$.
So $U \setminus \cl(B_1)$ is the union of the $\iota(V_k)$.
As each $V_k$ is regular open it follows that each $\iota(V_k)$ is regular open.
So $U \setminus \cl(B_1)$ is a finite union of definable regular open sets.
\end{proof}

\begin{prop}
\label{prop:H-interpret}
$\Sh H$ interprets $\Sa R$.
\end{prop}

\noindent We will need to apply the easy fact that if $\Sa M$ is an $\nip$ expansion of a linear order then every convex subset of $M$ is externally definable.

\begin{proof}
Let $O$ be the set of $a \in H$ such that $|a| < t$ for some $t > 0$.
Let $Q$ be the set of $a \in H$ such that $s < a < t$ for some $s,t > 0$.
Then $O$ and $Q$ are both convex, hence definable in $\Sh H$. \newline

\noindent Let $E$ be the equivalence relation on $H$ where $(a,b) \in E$ when $|a - b| < t$ for all $t > 0$.
We show that $E$ is definable in $\Sh H$.
Let $X$ be the set of $(a,b,c) \in N^3$ such that $| a - b | < c$.
Then $C := X \cap H^3$ is definable in $\Sa H$.
Observe that
$$ E = \bigcap_{c \in Q} \{ (a,b) \in H^2 : (a,b,c) \in C \} $$
so $E$ is definable in $\Sh H$. \newline

\noindent Each $E$-class is convex so we put a $\Sh H$-definable linear order on $H/E$ by declaring the class of $a$ to be less than the class of $b$ when $a < b$.
Observe that the $E$-class of any element of $O$ contains a unique real number and every real number is contained in the $E$-class of some element of $O$.
We therefore identify $O/E$ with $\R$, observe that the $\Sh H$-definable ordering on $O/E$ agrees with the usual order on $\R$, and let $\st : O \to \R$ be the quotient map.
As $\Sh H$ defines the usual order on $\R$, it defines a basis for the topology on $\R^n$.
We show that $\Sa R$ is a reduct of the structure induced on $\R$ by $\Sh H$.
By Lemma~\ref{lem:regular} it suffices to suppose that $U \subseteq \R^n$ is regular, open, and $\Sa R$-definable and show that $U$ is definable in $\Sh H$.
Let $U'$ be the subset of $N^n$ defined by any formula defining $U$.
It is easy to see that $\cl(U) = \st(U' \cap O^n)$, so $\cl(U)$ is definable in $\Sh H$.
So $U$ is $\Sh H$-definable as $U$ is the interior of $\cl(U)$ and $\Sh H$ defines a basis for $\R^n$.
\end{proof}

\noindent The proof of Proposition~\ref{prop:H-interpret} shows that $\Sa R$ is interpretable in the expansion of $\Sa H$ by two convex sets, $O$ and $Q$.

\bibliographystyle{abbrv}
\bibliography{H}
\end{document}